\documentclass{amsart}
\usepackage{amsfonts,amsmath,amstext,amssymb,amsthm, mathrsfs, amscd,enumerate,verbatim}
\usepackage[all]{xy}
\newtheorem{thm}{Theorem}[section]
\newtheorem{corollary}[thm]{Corollary}
\newtheorem{example}[thm]{Example}
\newtheorem{lemma}[thm]{Lemma}

\newtheorem{remark}[thm]{Remark}
\theoremstyle{definition}
\newtheorem{definition}[thm]{Definition}

 \numberwithin{equation}{section}
\newcommand\supp {\operatorname{Supp}}
\newcommand\ass {\operatorname{Ass}}
\newcommand\coass {\operatorname{Coass}}

\newcommand\att{\operatorname{Att}}
\newcommand{\cd}[2]{\operatorname{cd}_{#1}(#2)}
\newcommand{\cdd}[3]{\operatorname{cd}_{#1}(#2,#3)}
\newcommand{\grad}[3]{\operatorname{grade}_{#1}(#2,#3)}
\newcommand\depth{\operatorname{depth}}
\newcommand{\pd }{\operatorname{proj\,dim}}
\newcommand{\ann }{\operatorname{Ann}}
\newcommand{\ext}[4]{\operatorname{Ext}^{#1}_{#2}(#3,#4)}
\newcommand{\h}[3]{\operatorname{H}^{#1}_{#2}(#3)}
\newcommand{\gh}[4]{\operatorname{H}^{#1}_{#2}(#3,#4)}

 \newcommand\fa{\mathfrak a}
 \newcommand\fm{\mathfrak m}
\newcommand\fp{\mathfrak p}
\newcommand\fP{\mathfrak P}

\swapnumbers

\begin{document}

\title[Lichtenbaum-Hartshorne theorem]{Lichtenbaum-Hartshorne vanishing theorem for generalized local cohomology modules}%
(Originally published in Persian: A. Fathi, {\it Lichtenbaum-Hartshorne vanishing theorem for generalized local cohomology modules}, J. Adv. Math. Model. {\bf13}(2) (2023) 250--258. DOI:10.22055/JAMM.2023.43821.2166)

\author[A.~fathi]{Ali Fathi}
\address{Department of Mathematics, Zanjan Branch,
Islamic Azad University,  Zanjan, Iran.}
\email{alif1387@gmail.com}



\begin{abstract}
Let $R$ be a commutative Noetherian ring, and let $\mathfrak a$ be a proper  ideal of $R$. Let $M$ be a non-zero finitely generated $R$-module with the finite projective dimension $p$. Also, let $N$ be a non-zero finitely generated $R$-module with $N\neq\mathfrak{a} N$, and assume that $c$ is the greatest non-negative integer with the property that $\operatorname{H}^i_{\mathfrak a}(N)$, the $i$-th local cohomology module of $N$ with respect to $\fa$, is non-zero.  It is known that $\operatorname{H}^i_{\mathfrak a}(M, N)$, the $i$-th generalized local cohomology module  of $M$ and $N$ with respect to $\mathfrak a$, is zero for all  $i>p+c$. In this paper, we obtain the coassociated prime ideals of  $\operatorname{H}^{p+c}_{\mathfrak a}(M, N)$. Using this,  in the case when $R$ is a  local ring and $c$ is equal to the dimension of $N$, we give  a necessary and sufficient condition  for the vanishing of $\operatorname{H}^{p+c}_{\mathfrak a}(M, N)$ which  extends  the  Lichtenbaum-Hartshorne vanishing theorem for  generalized local cohomology modules.
 \end{abstract}
\maketitle
\setcounter{section}{0}
\section{\bf Introduction}
Throughout this paper, let $R$ be a commutative Noetherian ring with non-zero identity. Let $\fa$  be an ideal of $R$ and $N$ be an $R$-module. The $i$-th local cohomology module of $N$ with respect to $\fa$ was defined by Grothendieck as follows:
$$\h i{\fa}N:={\underset{n\in\mathbb N}{\varinjlim}\operatorname{Ext}^i_R(R/\fa^n,  N)};$$
see \cite{bs} for more details. For a pair of  $R$-modules $M$ and $N$, the $i$-generalized local cohomology module of $M, N$ with respect to $\fa$ was introduced by Herzog \cite{h} as follows:
$$\gh i{\fa}MN:=\underset{n\in\mathbb N}{\varinjlim}\operatorname{Ext}^i_R(M/\fa^n M,  N);$$
see \cite{h, b} for more details. It is clear that $\gh i{\fa}RN=\h i{\fa}N$. The cohomological dimension of $N$ with respect to $\fa$ and the cohomological dimension of $M, N$ with respect to $\fa$ are defined, respectively,  as follow:
$$\cd {\fa}N:=\sup\{i\in\mathbb N_0: \h i{\fa}N\neq 0\}$$
and
$$\cdd {\fa}MN:=\sup\{i\in\mathbb N_0: \gh i{\fa}MN\neq 0\}.$$

Assume that $N$ is finitely generated with finite dimension $d$. For each $i>d$, $\h i{\fa}N=0$ \cite[Theorem 6.1.2]{bs} (in other words,  $\cd {\fa}N\leq d$)
 and $\h d{\fa}N$ is Artinian  \cite[Exercise 7.1.7]{bs}.
When $R$ is local, Dibaei and Yassemi  as a main result proved in \cite[Theorem A]{dy}  that
$$\att_R(\h d{\fa}N)=\{\fp\in\ass_R(N): \cd{\fa}{R/\fp}=d\}.$$
This equality also  holds without the hypothesis that $R$ is local (see \cite[Theorem 2.5]{d}). Also, if $M$ is finitely generated with finite projective dimension $p$, then $\gh {i}{\fa}MN=0$  for all $i>p+d$ \cite[Lemma   5.1]{b} (i. e., $\cdd {\fa}MN\leq p+d$) and $\gh {p+d}{\fa}MN$  is Artinian (see for example \cite[Theorem 2.9]{maf} or \cite[Proposition 3.1]{hv}). When $R$ is local,  as a generalization of the theorem of  Dibaei and Yassemi, Gu and Chu show in \cite[Theorem 2.3]{gc} that
$$\att_R(\gh {p+d}{\fa}MN)=\{\fp\in\ass_R(N): \cdd {\fa}M{R/\fp}=p+d\}.$$
In \cite[Theorem 5.3]{f2015}, Fathi, Tehranian and Zakeri proved this equality in the case when $R$ is not necessarily local. They also show in \cite[Theorem 5.6]{f2015} that
\begin{equation}\tag{\dag}
\att_R(\gh {p+d}{\fa}MN)=\supp_R(\ext pRMR)\cap\att_R(\h d{\fa}N)
\end{equation}
whenever $R/ \ann_R(\h d{\fa}N)$ is a complete semilocal ring. This equality allows us to compute the set of attached prime ideals of the top generalized local cohomology module $\gh {p+d}{\fa}MN$ from the set of  attached prime ideals of the top local cohomology module $\h d{\fa}N$.

Now we set $c:=\cd {\fa}N$. For all $i>p+c$, we have $\gh i{\fa}MN=0$; see \cite[Proposition 2.8]{hv}. Since $c\leq d$,
$p+c$ yields   a sharper upper bound for $\cdd {\fa}MN$.  Note that $\h c{\fa}N$ and  $\gh {p+c}{\fa}MN$ are not necessarily Artinian.
In Theorem \ref{coass}, using the set of coassociated prime ideals of $\h c{\fa}N$, we compute the set of coassociated prime ideals of
$\gh {p+c}{\fa}MN$. More precisely, we show that
\begin{align*}
 &\coass_R\left(\gh{p+c}{\fa}MN\right)\\
 &=\{\fp\in\supp_R(M)\cap\coass_R\left(\h c{\fa}N\right): \pd_{R_\fp}(M_\fp)=p\}.
 \end{align*}

As a consequence of this equality, we prove in Corollary \ref{att}  that the equality (\dag) holds  even if    $R/ \ann_R(\h d{\fa}N)$ is not a complete semilocal ring,  and we show that
 \begin{align}
 \tag{\ddag}&\att_R\left(\gh{p+d}{\fa}MN\right)\\
 \nonumber& =\{\fp\in\supp_R(M)\cap\ass_R(N): \pd_{R_\fp}(M_\fp)=p, \cd{\fa}{R/\fp}=d\}.
 \end{align}
  In particular, if $R$ is local and $M$ is Cohen-Macaulay, then it is shown in  Corollary \ref{cm} that
 $$\att_R(\gh {p+d}{\fa}MN)=\{\fp\in\supp_R(M)\cap\ass_R(N): \cd {\fa}{R/\fp}=d\}.$$
 Finally, using the equality (\ddag), we extend  the Lichtenbaum-Hartshorne vanishing theorem for generalized local cohomology modules. More precisely, when $R$ is a local ring, we prove in Theorem \ref{lh} that  $\gh {p+d}{\fa}MN=0$ if and only if
   for all  $\fP\in\supp_{\widehat R}(\widehat M)\cap\ass_{\widehat R}(\widehat N)$
 with $\dim_{\widehat R}(\widehat R/\fP)=d$  and  $\pd_{\widehat R_\fP}(\widehat M_{\fP})=p$, $\dim_{\widehat R}(\widehat R/(\fa\widehat R+\fP))>0$.

\section{\bf Preliminaries}
 Let $M$  be an $R$-module. We denote the localization of $M$ at $\fp$ by $M_\fp$, and   the set of all prime ideals $\fp$  of $R$ such that
 $M_\fp$ is nonzero is called the support of $M$ and denoted by $\supp_R(M)$.  Also,  the annihilator of $M$ in $R$, denoted by $\ann_R(M)$,  is defined to be the set
 $\{r\in R: rx=0 \textrm{ for all } x\in M\}$.  If $\fp:=\ann_R(Rx)$ is a prime ideal of $R$ for some $x\in M$,
  then $\fp$ is called an associated prime ideal of $M$, and we denote  the set of all associated prime ideals of $M$ by $\ass_R(M)$.
   We will denote the set of all positive integers (respectively, non-negative integers)  by $\mathbb N$ (respectively, $\mathbb {N}_0$).

 The concepts of attached prime ideal and  secondary representation   as the dual of the concepts  of associated prime ideal and primary decomposition were introduced by Macdonald in \cite{m}.   An $R$-module $M$ is said to be secondary if $M\neq 0$ and, for each $r\in R$, the endomorphism $\mu_r:M\rightarrow M$ defined by  $\mu_r(x)=rx$ (for $x\in M$) is either surjective or nilpotent. If $M$ is secondary, then
 $\fp:=\sqrt{\ann_R(M)}$ is a prime ideal and $M$ is said to be $\fp$-secondary. A prime ideal $\fp$ is called an attached prime ideal of $M$ if $M$ has a $\fp$-secondary quotient. We denote the set of all attached prime ideals of $M$ by $\att_R(M)$.
If $M$ can be  written  as a finite sum of its secondary submodules, then we say that $M$ has a secondary representation. Such a secondary representation
$$M=M_1+\dots+M_t\quad  {\rm with }\ \sqrt{\ann_R(M_i)}=\fp_i\ {\rm for }\ i=1,\dots, t$$
  of $M$ is said to be minimal when none of the modules $M_i$ ($1\leq i\leq t$) is redundant and the prime ideals $\fp_1,\dots,\fp_t$ are distinct. Since the sum of two $\fp$-secondary submodules of $M$
is again  $\fp$-secondary,  so if $M$ has a secondary representation, then it has a minimal one.  When the above secondary representation is minimal, then $\att_R(M)=\{\fp_1,\dots, \fp_t\}$, and hence $t$ and the set  $\{\fp_1,\dots, \fp_t\}$ are independent of the choice of minimal secondary representation of $M$.  Artinian modules have secondary representation.

Yassemi \cite{y}  has introduced the coassociated prime ideal as a dual of associated prime ideal. In Yassemi's  definition, we  do not  need  to assume  that the module has a secondary representation, and if a module has a secondary representation, then its sets of coassociated prime ideals and attached prime ideals are same (see \cite[Theorem 1.14]{y}).
\begin{definition}
We say that an $R$-module $M$ is  cocyclic when $M$ is a submodule of $E(R/\fm)$ for some maximal ideal $\fm$ of $R$, where $E(R/\fm)$ denotes the injective envelope of $R/\fm$.
\end{definition}
\begin{definition}
We say that a prime ideal $\fp$ of $R$ is a coassociated prime ideal of an $R$-module $M$ when there exists a cocyclic homomorphic image $L$ of $M$ such that $\fp=\ann_R(L)$. We denote by $\coass_R(M)$ the set of all coassociated prime ideals of $M$.
\end{definition}
 \section{\bf Main results}
 In the following theorem, using the set of coassociated prime ideals  of  the top  local cohomology module, we compute the set of coassociated prime ideals of  the top generalized local cohomology module.
 \begin{thm}\label{coass}
 Let $\fa$ be an ideal of $R$ and $M$ be a non-zero finitely generated $R$-module with finite projective dimension $p$. Let $N$ be an $R$-module such that  $N\neq \fa N$ and  $c:=\cd {\fa}N$. Then for each $n>p+c$, $\gh n{\fa}MN=0$ and
  $$\gh {p+c}{\fa}MN\cong\ext pRMR\otimes_R\h c{\fa}N.$$
  In particular,
   \begin{align*}
 &\coass_R\left(\gh{p+c}{\fa}MN\right)\\
 &=\{\fp\in\supp_R(M)\cap\coass_R\left(\h c{\fa}N\right): \pd_{R_\fp}(M_\fp)=p\}.
 \end{align*}
  \end{thm}
\begin{proof}
Hassanzadeh and Vahidi, in \cite[Proposition 2.8]{hv}, show  that $\gh n{\fa}MN=0$ for all $n>p+c$ and
$$\gh {p+c}{\fa}MN\cong \ext pRM{\h c{\fa}N}.$$

Now the functor  $\ext pRM{\cdot}$ is  additive and right exact. Also since $M$ is a finitely generated module over a Noetherian ring, it follows from \cite[Lemma 3.1.16]{e} that
$\ext pRM{\cdot}$ preserves direct sums and so \cite[Theorem 5.45]{ro} implies that $\ext pRM{\cdot}\cong \ext pRMR\otimes_R(\cdot).$
Therefore
 $$\gh {p+c}{\fa}MN\cong\ext pRMR\otimes_R\h c{\fa}N.$$
Hence, by \cite[Theorem 1.21]{y}, we have
 $$\coass_R\left(\gh{p+c}{\fa}MN\right) =\supp_R\left(\ext pRMR\right)\cap\coass_R\left(\h c{\fa}N\right).$$
Thus to complete the proof, it is sufficient for us to show that
$$\supp_R(\ext pRMR)=\{\fp\in\supp_R(M): \pd_{R_\fp}(M_\fp)=p\}.$$
Suppose that
$\fp\in\supp_R(\ext pRMR)$.
Since $M$  is a finitely generated module over the Noetherian ring $R$, it follows from \cite[Proposition 7.39]{ro} that
$$\ext p{R_\fp}{M_\fp}{R_\fp}\cong(\ext pRMR)_\fp\neq 0.$$
Therefore $\fp\in\supp_R(M)$ and $\pd_{R_\fp}(M_\fp)\geq p$.  Thus  $\pd_{R_\fp}(M_\fp)=p$.
Conversely, assume that $\fp\in\supp_R(M)$ and $\pd_{R_\fp}(M_\fp)=p$. Therefore  $\ext p{R_\fp}{M_\fp}{R_\fp}\neq 0$ by
\cite[Section 19, Lemma 1(iii)]{mat}. Hence $(\ext pRMR)_\fp\neq 0$, and this completes the proof.
\end{proof}
Let the notations and assumptions be as in Corollary \ref{att}.
 Fathi, Tehranian and Zakeri, in \cite[Theorem 5.6]{f2015}, proved that
\begin{align}\tag{\dag}
\att_R(\gh {p+d}{\fa}MN)=\supp_R(\ext pRMR)\cap\att_R(\h d{\fa}N)
\end{align}
 whenever $B:=R/ \ann_R(\h d{\fa}N)$ is a complete semilocal ring.
In the following corollary it is shown that the equality (\dag)  holds without the hypothesis  that  $B$ is a complete semilocal ring.
\begin{corollary}\label{att}
Let $\fa$ be an ideal of $R$ and let $M, N$ be non-zero finitely generated $R$-modules such that   $p:=\pd_R(M)<\infty$    and     $d:=\dim_R(N)<\infty$.
Then   $\gh{p+d}{\fa}MN$ is Artinian and
  \begin{align*}
 \tag{\ddag} &\att_R(\gh{p+d}{\fa}MN)\\
  &=\{\fp\in\supp_R(M)\cap\ass_R(N) : \pd_{R_\fp}( M_\fp)=p, \cd{\fa}{R/\fp}=d \}.
  \end{align*}
   \end{corollary}
    \begin{proof}
We set $c:=\cd{\fa}N$. By Grothendieck's vanishing theorem \cite[Theorem 6.1.2]{bs}, $c\leq d$. Now if $c<d$, then by the previous theorem
$\gh{p+d}{\fa}MN=0$, and so  its set of attached prime ideals is empty. On the other hand, by \cite[Theorem 1.2]{dnt}, for each
$\fp\in\ass_R(N)$, we have $\cd{\fa}{R/\fp}\leq c<d$. Hence the set in the right hand side of   (\ddag) is also empty and so the equality (\ddag)
holds in this case. We may therefore assume that $c=d$. By \cite[Exersise 7.1.7]{bs},  $\h d{\fa}N$  is Artinian and since $\ext pRMR$ is finitely generated,  $\ext pRMR\otimes_R\h d{\fa}N$ is Artinian.  By Theorem \ref{coass}, this module is isomorphic to $\gh {p+d}{\fa}MN$, and so $\gh {p+d}{\fa}MN$  is Artinian (Artinianess of $\gh {p+d}{\fa}MN$ is not a new result, see for example \cite[Theorem 2.9]{maf} or \cite[Proposition 3.1]{hv}).
Hence the  sets of attached prime ideals of $\h d{\fa}N$  and $\gh {p+d}{\fa}MN$ coincide with their sets of coassociated prime ideals \cite[Theorem 1.14]{y}. Therefore the equality (\ddag) follows from \cite[Theorem 2.5]{d} and the last part of Theorem \ref{coass}.
\end{proof}

\begin{corollary}\label{cm}
Let $R$ be a local ring and $\fa$ be an ideal of $R$. Let $M, N$ be non-zero finitely generated $R$-modules such that $M$ is Cohen-Macaulay,
$p:=\pd_R(M)<\infty$ and $d:=\dim_R(N)$. Then we have
  \begin{align*}
  \att_R\gh{p+d}{\fa}MN
  =\{\fp\in\supp_R(M)\cap\ass_R(N) :  \cd{\fa}{R/\fp}=d \}.
  \end{align*}
   \end{corollary}
 \begin{proof}
  By \cite[Corollary 9.46, Remark 9.4.8(a)]{bh} and the Auslander–-Buchsbaum formula  \cite[Theorem 1.3.3]{bh}, we have
$$\dim_R(R)\leq\pd_R(M)+\dim_R(M)=\depth_R(R).$$
Thus $R$ is Cohen-Macaulay. If $\fp\in\supp_R(M)$, then  $\dim_R(R/\fp)=\dim_R(M/\fp M)$, and so the Auslander-Buchsbaum formula
and \cite[Theorem 2.1.3(b)]{bh} imply that
\begin{align*}
\pd_{R_\fp}(M_\fp)&=\dim_{R_\fp} ({R_\fp})-\dim_{R_\fp}( {M_\fp})\\
&=(\dim_R(R)-\dim_R(R/\fp))-(\dim_R(M)-\dim_R(M/\fp M))\\
&=\dim_R(R)-\dim_R(M)\\
&=p.
\end{align*}
Now the assertion follows from the previous corollary.
\end{proof}
In Theorem \ref{lh}, we are going to prove the Lichtenbaum-Hartshorne vanishing theorem for generalized local cohomology modules.
Before that the following lemma which  extends the flat base change theorem \cite[Theorem 4.3.2]{bs} for generalized local cohomology modules is needed. This lemma is stated in \cite[Lemma 2.1(ii)]{dst} without  proof. Here we give a proof for  the  readers' convenience.
Also, we note that in our proof we need to assume that the first module in the generalized local cohomology module is finitely generated but  in \cite[Lemma 2.1(ii)]{dst} there is no  such a   restriction on the module.

\begin{lemma}[{\cite[Lemma 2.1(ii)]{dst}}]\label{flat}
 Let $\fa$ be an ideal of $R$, $M$ be a finitely generated $R$-module and $N$ be an arbitrary $R$-module. Let $B$ be a flat $R$-algebra. Then for each $i\in\mathbb N_0$, we have
 $$B\otimes_R\gh i{\fa}MN\cong \gh i{\fa B}{B\otimes_RM}{B\otimes_RN}.$$
\end{lemma}
\begin{proof}
For each $n\in\mathbb N_0$, we have
\begin{align*}
B\otimes_R(M/\fa^n M)&\cong B\otimes_RM\otimes_R(R/\fa^n)\\
&\cong(B\otimes_RM)\otimes_BB\otimes_R(R/\fa^n)\\
&\cong(B\otimes_RM)\otimes_B(B/\fa^n B)\\
&\cong(B\otimes_RM)\otimes_B(B/(\fa B)^n)\\
&\cong(B\otimes_RM)/(\fa B)^n(B\otimes_R M).
\end{align*}
Now it follows from \cite[Theorem 5.27]{ro}, \cite[Theorem 3.2.5]{e} and the above isomorphism that
\begin{align*}
B\otimes_R\gh i{\fa}MN&\cong B\otimes_R\underset{n}{\varinjlim}\ext iR{M/\fa^nM}N\\
&\cong\underset{n}{\varinjlim}(B\otimes_R\ext iR{M/\fa^nM}N)\\
&\cong\underset{n}{\varinjlim}(\ext iB{(B\otimes_RM)/(\fa B)^n(B\otimes M)}{B\otimes_RN}\\
&\cong \gh i{\fa B}{B\otimes_RM}{B\otimes_RN}.
\end{align*}
\end{proof}

\begin{thm}[The Lichtenbaum-Hartshorne vanishing theorem for generalized local cohomology modules]\label{lh}
Let $(R, \fm)$  be a local ring and $\fa$ be a proper ideal of $R$. Let $M, N$ be  non-zero finitely generated $R$-modules such that
 $p:=\pd_R(M)<\infty$ and  $d:=\dim_R (N)$. Then the following statements are equivalent:
   \begin{enumerate}[\rm(i)]
   \item $\gh {p+d}{\fa}MN=0$;
   \item for each $\fP\in\supp_{\widehat R}(\widehat M)\cap\ass_{\widehat R}(\widehat N)$
    satisfying     $\pd_{\widehat R_\fP}(\widehat M_{\fP})=p$ and $\dim_{\widehat R}(\widehat R/\fP)=d$, we have  $\dim_{\widehat R}(\widehat R/(\fa\widehat R+\fP))>0$.
        \end{enumerate}
  \end{thm}
\begin{proof}
 $\widehat R$ is a Noetherian local ring with maximal ideal  $\fm \widehat R$ (see \cite[Theorem 8.12]{mat}) and for a finitely generated $R$-module $L$ we have $L\otimes_R\widehat{R}=\widehat{L}$ (see \cite[Theorem 8.7]{mat}). Since  $\widehat R$ is a flat $R$-algebra (see \cite[Theorem 8.8]{mat}), by Lemma
 \ref{flat}, we have
  $${\widehat {R}}\otimes_R\gh i{\fa}MN\cong\gh i{\fa\widehat R}{\widehat M}{\widehat N}.$$
 The above isomorphism implies that $\gh i{\fa}MN=0$
 if and only if $\gh i{\fa\widehat R}{\widehat M}{\widehat N}=0$ because  $\widehat R$ is a faithfully flat $R$-module (see \cite[Theorem 8.14]{mat}).  Also, for a finitely generated $R$-module $L$, by \cite[Corollary 2.1.8(a)]{bh},
 we have
 $$\depth_{\widehat R}(\widehat L)=\depth_R(L),\ \dim_{\widehat R}(\widehat L)=\dim_R(L).$$
 By hypothesis,  $M$ has a finite free resolution. Tensoring a finite free  resolution of $M$ by $\widehat R$ yields  a finite free resolution for $\widehat M$ over $\widehat R$, and so $\pd_{\widehat R}(\widehat M)<\infty$. Now, by the above  equality and the Auslander-Buchsbaum formula, we obtain
  $$\pd_{\widehat R}(\widehat M)=\pd_R(M).$$
  Therefore  we can (and  do) replace  $M$,  $N$ and $R$
 by $\widehat M$,  $\widehat N$ and $\widehat R$ and assume henceforth in this proof that $R$ is complete.

  (i)$\Rightarrow$(ii). Assume that $\gh {p+d}{\fa}MN=0$ and $\fp\in\supp_R(M)\cap\ass_R(N)$ is such that  $\dim_R(R/\fp)=d$ and  $\pd_{R_\fp}( M_{\fp})=p$. Since   $\att_R(\gh {p+d}{\fa}MN)=\emptyset$, by Corollary \ref{att}, we have  $\cd {\fa}{R/\fp}<d$.
It follows from the independence theorem \cite[Theorem 4.2.1]{bs} that  $\h d{\fa+\fp}{R/\fp}\cong\h d{\fa}{R/\fp}=0$. Since  $\dim_R(R/\fp)=d$,  Grothendieck's non-vanishing theorem \cite[Theorem 6.1.4]{bs} implies that the proper ideal $\fa+\fp$ is not $\fm$-primary and so   $\dim_R(R/(\fa+\fp))>0$.

 (ii)$\Rightarrow$(i). Assume that the condition (ii) holds. If $\gh {p+d}{\fa}MN\neq 0$, then   $\att_R(\gh {p+d}{\fa}MN)$ is not empty
and hence, by Corollary \ref{att}, there exists $\fp\in\supp_R(M)\cap\ass_R(N)$ such that    $\cd {\fa}{R/\fp}=d$ and  $\pd_{R_\fp}(M_{\fp})=p$.
By Grothendieck's vanishing theorem \cite[Theorem 6.1.2]{bs}, $\cd {\fa}{R/\fp}\leq \dim_R(R/\fp)$. Thus $\dim_R(R/\fp)=d$ and the statement (ii) yields  $\dim_R(R/(\fa+\fp))>0$.  We set $\bar{R}:=R/\fp$. Now $\bar R$ is a complete local domain of dimension $d$ and $\dim_{\bar R}(\bar R/\fa\bar R)=\dim_R(R/(\fa+\fp))>0$.
Thus   the Lichtenbaum-Hartshorne vanishing theorem (see \cite[Theorem 8.2.1 ]{bs}) for the ring $\bar R$  implies that $\h d{\fa\bar R}{\bar R}=0$. Hence, by the independence theorem, $\h d{\fa}{R/\fp}\cong\h d{\fa\bar R}{\bar R}=0$. Since $\cd{\fa}{R/\fp}\leq\cd{\fa}N\leq d$, we obtain $\cd {\fa}{R/\fp}<d$, which is a contradiction. Therefore
     $\gh {p+d}{\fa}MN=0$.
\end{proof}

\begin{remark}
Let  $(R, \fm)$ be a local ring. Let $M, N$ be  non-zero finitely generated $R$-modules such that $p:=\pd_R(M)<\infty$ and $d:=\dim_R(N)$.
By Grothendieck's vanishing and non-vanishing theorems \cite[Theorems 6.1.2 and 6.1.4]{bs}, we have  $\cd {\fm}N=\dim_R(N)$. The exact value of  $\cdd {\fm}MN$  is unknown under the above assumptions. However, if in addition  $R$ is Cohen-Macaulay, then Divaani-Aazar and Hajikarimi in \cite[Theorem 3.5]{dh} proved  that
$$\cdd {\fm}MN=\dim_R(R)-\grad R{\ann_R(N)}M.$$
We know that $p+d$  is an upper bound for $\cdd {\fm}MN$. If we set   $\fa:=\fm$ in Theorem \ref{lh}, then it is not true to say that
since  $\dim_{\widehat R}(\widehat R/\fm\widehat R+\fP)=0$ for each prime ideal $\fP$ of  $\widehat R$, $\gh {p+d}{\fm}MN$ is non-zero and so $\cdd {\fm}MN=p+d$. The following example shows that  $p+d$ can be a strict upper bound for $\cdd {\fm}MN$.
In fact, when  there is not a prime ideal $\fP$ in $\supp_{\widehat R}(\widehat M)\cap\ass_{\widehat R}(\widehat N)$ satisfying
 $\dim_{\widehat R}(\widehat R/\fP)=d$ and  $\pd_{\widehat R_\fP}(\widehat M_{\fP})=p$, then the statement  (ii) in Theorem \ref{lh} is true and hence $\gh {p+d}{\fm}MN=0$.
\end{remark}

\begin{example}
Let $K$ be a field and $R:=K[[x, y]]$ be the ring of formal power series over $K$ in indeterminates $x,y$.
$R$ is a complete regular local  ring of dimension 2 with maximal ideal $\fm:=(x, y)$. We set  $M:=R/(x^2, xy)$. It follows from $\ass_R(M)=\{(x), (x, y)\}$ that $\depth_R(M)=0$ and $\dim_R(M)=\dim_R(R/(x))=1$. Since $R$ is regular, all modules have  finite projective dimension and  so the Auslander-Buchsbaum formula gives  $\pd_R(M)=2$. Therefore  $\pd_R(M)+\dim_R(R)=4$. Now since  $\ass_R(R)=\{0\}$,
$\supp_R(M)\cap\ass_R(R)=\emptyset$ and so, by Theorem \ref{lh} or Corollary \ref{att}, we obtain  $\gh 4{\fm}MR=0$. Hence
 $$\cdd {\fm}MR<\pd_R(M)+\dim_R(R).$$
 Furthermore, since $R$ is Cohen-Macaulay and $M$ has a finite projective dimension,  the Divaani Azar--Hajikarimi formula implies that
 $$\cdd {\fm}MR=\dim_R(R)-\grad R{\ann_R(R)}M=2-0=2.$$
\end{example}

\section*{Acknowledgements}
 The author would like to thank the referees for careful reading of the Persian version of  this paper
and for helpful suggestions.

\end{document}